\documentclass[12pt]{amsart}
\usepackage{amssymb,latexsym,eufrak,amsmath,amscd, graphicx}
\usepackage[dvipsnames,usenames]{color}

\usepackage{amsfonts}
\usepackage[all]{xypic}
\UseTips

\renewcommand{\to}[1][]{\xrightarrow{\ #1\ }}

\newcommand{\forget}[1]{}  

\makeatletter
\renewcommand{\theenumi}{\@roman\c@enumi}
\makeatother

\renewcommand{\phi}{\varphi}
\renewcommand{\epsilon}{\varepsilon}
\renewcommand{\theta}{\vartheta}

\newcommand{\llbracket}{[\negthinspace[}
\newcommand{\rrbracket}{]\negthinspace]}

\def\ZZ{{\mathbf Z}}

\def\QQ{{\mathbf Q}}

\def\cF{\mathcal{F}}

\def\cO{\mathcal{O}}

\def\o{\circ}

\def\.{\cdot}
\def\({\Big{(}}
\def\){\Big{)}}
\def\^{\widehat}
\def\~{\widetilde}
\def\*{{}^*\!}
\def\[{\llbracket}
\def\]{\rrbracket}

\renewcommand{\and}{ \quad \text{and} \quad }

\DeclareMathOperator{\Pic}{Pic}
\DeclareMathOperator{\Cart}{Cart}

\DeclareMathOperator{\Bs}{Bs}
\DeclareMathOperator{\SB}{\mathbf{SB}}
\DeclareMathOperator{\BB}{\mathbf{B}_{+}}

\newtheorem{lemma}{Lemma}[section]
\newtheorem{theorem}[lemma]{Theorem}
\newtheorem{corollary}[lemma]{Corollary}

\theoremstyle{definition}

\newtheorem{remark}[lemma]{Remark}

\theoremstyle{remark}
\newtheorem*{remark*}{Remark}
\newtheorem*{note*}{Note}

\begin{document}

\title{The augmented base locus in positive characteristic}

\thanks{2010\,\emph{Mathematics Subject Classification}.  Primary 14A15; Secondary 14E99.
  \newline The first author was partially supported by an EPSRC grant, the second author
  was partially supported by NSF research grant no: 0701101, and the third author was
  partially supported by NSF research grant no: 1068190 and by a Packard Fellowship.}


\author[P.~Cascini]{Paolo~Cascini}
\address{Department of Mathematics, Imperial College London, London SW7 2AZ, UK}
\email{{p.cascini@imperial.ac.uk}}

\author[J.~M$^{\rm c}$Kernan]{James~M\textsuperscript{c}Kernan}
\address{Department of Mathematics, MIT, 77 Massachusetts Avenue, 
Cambridge, MA 02139, USA}
\email{{mckernan@math.mit.edu}}

\author[M.~Musta\c{t}\u{a}]{Mircea~Musta\c{t}\u{a}}
\address{Department of Mathematics, University of Michigan,
Ann Arbor, MI 48109, USA}
\email{{mmustata@umich.edu}}

\dedicatory{To Slava Shokurov, with admiration, on the~occasion of
his~sixtieth~birthday}

\begin{abstract}
Let $L$ be a nef line bundle on a projective scheme $X$ in positive characteristic.  We
prove that the augmented base locus of $L$ is equal to the union of the irreducible closed
subsets $V$ of $X$ such that $L\vert_V$ is not big. For a smooth variety in characteristic
zero, this was proved by Nakamaye using vanishing theorems.
\end{abstract}

\maketitle

\markboth{P.~CASCINI, J.~M\textsuperscript{c}KERNAN, AND M.~MUSTA\c{T}\u{A}}{THE AUGMENTED BASE LOCUS IN POSITIVE CHARACTERISTIC}

\section{Introduction}

Let $X$ be a projective scheme over an algebraically closed field $k$, and $L$ a line
bundle on $X$. The base locus $\Bs(L)$ of $L$ is the closed subset of $X$ consisting of
those $x\in X$ such that every section of $L$ vanishes at $x$.  It is easy to see that if
$m_1$ and $m_2$ are positive integers such that $m_1$ divides $m_2$, then $\Bs
(L^{m_2})\subseteq \Bs(L^{m_1})$.  It follows from the Noetherian property that $\Bs(L^m)$
is independent of $m$ if $m$ is divisible enough; this is the \emph{stable base locus}
$\SB(L)$ of $L$.

The stable base locus is a very interesting geometric invariant of $L$, but it is quite
subtle: for example, there are numerically equivalent Cartier divisors whose stable base
loci are different.  Nakamaye introduced in \cite{Nakamaye} the following upper
approximation of $\SB(L)$, the \emph{augmented base locus} $\BB(L)$. If $L\in\Pic(X)$ and
$A\in \Pic(X)$ is ample, then
$$\BB(L):=\SB(L^m\otimes A^{-1}),$$
for $m\gg 0$. It is easy to check that this is well-defined, it is independent of $A$, and only depends on the numerical equivalence class of $L$.
The following is our main result.

\begin{theorem}\label{thm_main}
Let $X$ be a projective scheme over an algebraically closed field of positive characteristic.
If $L$ is a nef line bundle on $X$, then $\BB(L)$ is equal to $L^{\perp}$, the union of all
irreducible closed subsets $V$ of $X$ such that $L\vert_V$ is not big.
\end{theorem}

We note that since $L$ is nef, for an irreducible closed subset $V$ of $X$, the restriction
$L\vert_V$ is not big if and only if $V$ has positive dimension and $(L\vert_V^{\dim(V)})=0$.
When $X$ is a smooth projective variety in characteristic zero, the above theorem was proved in
\cite{Nakamaye}, making use of the Kawamata-Viehweg vanishing theorem. 
It is an interesting question whether the result holds in characteristic zero when the variety is singular. 

The proof of Theorem~\ref{thm_main} makes use in an essential way of the Frobenius morphism.
The following is a key ingredient in the proof.

\begin{theorem}\label{thm_main2}
Let $X$ be a projective scheme over an algebraically closed field of positive characteristic.
If $L$ is a nef line bundle on $X$ and $D$ is an effective Cartier divisor such that
$L(-D)$ is ample, then
$\BB(L)=\BB(L\vert_D)$.
\end{theorem}

In the proofs of the above results we make use of techniques introduced by Keel in \cite{Keel}. In fact,
if we replace in Theorem~\ref{thm_main2} the two augmented base loci by the corresponding
stable base loci, we recover one of the main results in \cite{Keel}. We give a somewhat simplified proof of this result (see Corollary~\ref{cor2_thm_Keel} below), and this proof extends to give also
Theorem~\ref{thm_main2}.

In the next section we recall some basic facts about 
augmented base loci. The proofs of Theorems~\ref{thm_main2} and \ref{thm_main} are then given in \S 3. 

\subsection{Acknowledgment}
We are indebted to Rob Lazarsfeld for discussions that led to some of the results in this
paper.  We would also like to thank Se\'an Keel for several very useful discussions and the
referee for some useful comments.

\section{Augmented base loci and big line bundles}

In this section we review some basic facts about the augmented base locus. This notion is
usually defined for integral schemes. However, even if one is only interested in this
restrictive setting, for the proof of Theorem~\ref{thm_main} we need to also consider
possibly reducible, or even non-reduced schemes. We therefore define the augmented base
locus in the more general setting that we will need. Its general properties follow as in
the case of integral schemes, for which we refer to \cite{ELMNP}.

Let $X$ be a projective scheme over an algebraically closed field $k$. If $L$ is a line
bundle on $X$ and $s\in H^0(X,L)$, then we denote by $Z(s)$ the zero-locus of $s$ (with
the obvious scheme structure). Note that $Z(s)$ is defined by a locally principal ideal,
but in general it is not an effective Cartier divisor (if $X$ is reduced, then $Z(s)$ is
an effective Cartier divisor if and only if no irreducible component of $X$ is contained
in $Z(s)$). The base locus of $L$ is by definition the closed subset of $X$ given by
$$\Bs(L):=\bigcap_{s\in H^0(X,L)}Z(s)_{\rm red}.$$

If $m$ is a positive integer and $s\in H^0(X,L)$, then it is clear that
$Z(s)_{\rm red}=Z(s^{\otimes m})_{\rm red}$, hence $\Bs(L^m)\subseteq \Bs(L)$. 
More generally, we have $\Bs(L^{mr})\subseteq \Bs(L^r)$ for every $m,r\geq 1$,
hence by the Noetherian property there is $m_0\geq 1$ such that
$$\SB(L):=\bigcap_{r\geq 1}\Bs(L^r)$$
is equal to $\Bs(L^m)$ whenever $m$ is divisible by $m_0$. The closed subset
$\SB(L)$ of $X$ is the \emph{stable base locus} of $L$. It follows by definition that
$\SB(L)=\SB(L^r)$ for every $r\geq 1$. 

Since $X$ is projective, every line bundle is of the form $\cO_X(D)$, for some Cartier divisor 
$D$ (see \cite{Nakai}). We will sometimes find it convenient to work with Cartier divisors, rather than line bundles.  
Let $\Cart(X)_{\QQ}:=\Cart(X)\otimes_{\ZZ}\QQ$ denote the group of Cartier
$\QQ$-divisors and $\Pic(X)_{\QQ}:=\Pic(X)\otimes_{\ZZ}\QQ$.
For a Cartier divisor $D$ , we put $\SB(D)=\SB(\cO_X(D))$. 
Since $\SB(D)=\SB(rD)$ for every $r\geq 1$, the definition extends in the obvious way to
$\Cart(X)_{\QQ}$. 

Given a Cartier $\QQ$-divisor $D$, the augmented base locus of $D$ is
$$\BB(D):=\bigcap_A\SB(D-A),$$
where the intersection is over all ample Cartier $\QQ$-divisors on $X$.
It is easy to see that if $A_1$ and $A_2$ are ample Cartier $\QQ$-divisors such that
$A_1-A_2$ is ample, then $\SB(D-A_2)\subseteq\SB(D-A_1)$. It follows from the Noetherian
property that there is an ample Cartier $\QQ$-divisor $A$ such that
$\BB(D)=\SB(D-A)$. Furthermore, in this case if $A'$ is ample and $A-A'$ is ample, too,
then $\BB(D)=\SB(D-A')$. It is then clear that if $H$ is any ample Cartier divisor on $X$, then
for $m\gg 0$ we have
$$\BB(D)=\SB\left(D-\frac{1}{m}H\right)=\SB(mD-H).$$

The following properties of the augmented base locus are direct consequences of the definition
(see \cite[\S 1]{ELMNP}).
\begin{enumerate}
\item[1)] For every Cartier $\QQ$-divisor $D$, we have $\SB(D)\subseteq\BB(D)$.
\item[2)] If $D_1$ and $D_2$ are numerically equivalent Cartier $\QQ$-divisors, then 
$\BB(D_1)=\BB(D_2)$.
\end{enumerate}

If $D$ is a Cartier divisor and $L=\cO_X(D)$, we also write $\BB(L)$ for $\BB(D)$. 

\begin{lemma}\label{lem1}
If $L$ is a line bundle on the projective scheme $X$, and $Y$ is a closed subscheme 
of $X$, then
\begin{enumerate}
\item[i)] $\SB(L\vert_Y)\subseteq\SB(L)$.
\item[ii)] $\BB(L\vert_Y)\subseteq \BB(L)$.
\end{enumerate}
\end{lemma}

\begin{proof}
The first assertion follows from the fact that if $s\in H^0(X,L)$, then $Z(s\vert_Y)\subseteq Z(s)$,
hence $\Bs(L^m\vert_Y)\subseteq \Bs(L^m)$ for every $m\geq 1$. 
For the second assertion, fix an ample line bundle $A$ on $X$, and let $m\gg 0$ be such that
$\BB(L)=\SB(L^m\otimes A^{-1})$. Since $A\vert_Y$ is ample on $Y$, using i) and the definition
of the augmented base locus of $L\vert_Y$, we obtain
$$\BB(L\vert_Y)\subseteq \SB(L^m\vert_Y\otimes A^{-1}\vert_Y)\subseteq\SB(L^m\otimes A^{-1})
=\BB(L).$$
\end{proof}

Recall that a line bundle $L$ on an integral $n$-dimensional scheme $X$ is \emph{big} if there is $C>0$ such that
$h^0(X,L^m)\geq Cm^{n}$ for $m\gg 0$. Equivalently, this is the case if and only if there are Cartier
divisors $A$ and $E$, with $A$ ample and $E$ effective, such that $L^m\simeq
\cO_X(A+E)$ for some $m\geq 1$. We refer to \cite[\S 2.2]{positivity} for basic facts about big 
line bundles on integral schemes.
The following lemma is well-known, but we include a proof for completeness.

\begin{lemma}\label{lem2}
Let $X$ be an $n$-dimensional projective scheme and $L$ a line bundle on $X$.
For every coherent sheaf $\cF$ on $X$, there is $C>0$ such that $h^0(X,\cF\otimes L^m)\leq Cm^n$ for every $m\geq 1$.
\end{lemma}

\begin{proof}
Let us write $L\simeq A\otimes B^{-1}$ for suitable very ample line bundles $A$ and $B$.
For every $m\geq 1$, the line bundle $B^m$ is very ample. By choosing a section
$s_m\in H^0(X,B^m)$ such that $Z(s_m)$ does not contain any of the associated subvarieties of $\cF$, we obtain an inclusion $H^0(X,\cF\otimes L^m)\hookrightarrow
H^0(X,\cF\otimes A^m)$. Since $h^0(X,\cF\otimes A^m)=P(m)$ for $m\gg 0$, where $P$
is a polynomial of degree $\leq n$, we obtain the assertion in the lemma.
\end{proof}

If $X$ is reduced, and $A$, $D$ are Cartier divisors on $X$ with $A$ ample and
$D$ effective, then the restriction of $\cO_X(A+D)$ to every irreducible component $Y$ of $X$ is big (note that 
the restriction $D\vert_{Y}$ is well-defined and gives an effective divisor on $Y$).
As a consequence of the next lemma, we will obtain a converse to this statement.

\begin{lemma}\label{lem3}
Let $X$ be a reduced projective scheme. Given line bundles $L$ and $A$ on $X$,
with $A$ ample, if
 $m\gg 0$ and $s\in H^0(X, L^m\otimes A^{-1})$ is general, 
then for every irreducible component $Y$ of $X$ such that  
 $L\vert_Y$ is big, we have $Y\not\subseteq Z(s)$. 
\end{lemma}

Note that since $A$ is ample, if $s\in H^0(X,L^m\otimes A^{-1})$
and $Y'$ is an irreducible component of $X$ (considered with the reduced scheme structure) such that $L\vert_{Y'}$ is not big, then $Y'\subseteq Z(s)$.

\begin{proof}[Proof of Lemma~\ref{lem3}]
Suppose that $Y$ is an irreducible component of $X$ (considered with the reduced
structure) such that $L\vert_Y$ is big, but such that for infinitely many $m$ we have
$Y\subseteq Z(s)$ for every $s\in H^0(X,L^m\otimes A^{-1})$.  If $W$ is the union of the
other irreducible components of $X$, also considered with the reduced scheme structure,
then we have an exact sequence
$$0\to \cO_X\to\cO_Y\oplus\cO_W\to\cO_{Y\cap W}\to 0,$$
where $Y\cap W$ denotes the (possibly non-reduced) scheme-theoretic intersection
of $Y$ and $W$. After
tensoring with $L^m\otimes A^{-1}$ and taking global sections, this induces the exact sequence
$$0\to H^0(X, L^m\otimes A^{-1})\to H^0(Y,L^m\otimes A^{-1}\vert_Y)\oplus
H^0(W,L^m\otimes A^{-1}\vert_W)$$
$$\to H^0(Y\cap W,L^m\otimes A^{-1}\vert_{Y\cap W}).$$
By assumption, the map $H^0(X, L^m\otimes
A^{-1})\to H^0(Y,L^m\otimes A^{-1}\vert_Y)$ is zero for infinitely many $m$, in which case
the above exact sequence implies
\begin{equation}\label{eq_lem3}
h^0(Y,L^m\otimes A^{-1}\vert_Y)\leq h^0(T, L^m\otimes A^{-1}\vert_{T}).
\end{equation}
Let $n=\dim(Y)$. 
Since  $\dim(T)\leq n-1$, it follows from Lemma~\ref{lem2} that we can find $C>0$ such that 
$$h^0(T, L^m\otimes A^{-1}\vert_{T})\leq C m^{n-1}$$
for all $m$. On the other hand, since $L\vert_Y$ is big, it is easy to see that there is $C'>0$ such that
$h^0(Y,L^m\otimes A^{-1}\vert_Y)\geq C'm^n$ for all $m\gg 0$. These two estimates
contradict (\ref{eq_lem3}) when $m\gg 0$.
\end{proof}

\begin{corollary}
Let $L$ be a line bundle on the reduced projective scheme $X$. If the restriction of $L$
to every irreducible component of $X$ is big, then for every ample line bundle $A$ and every
$m\gg 0$, the zero locus of a  general section in $H^0(X,L^m\otimes A^{-1})$ defines an
effective Cartier divisor on $X$. 
\end{corollary}

\section{Main results}

In this section we assume that all our schemes are of finite type over an algebraically
closed field $k$ of characteristic $p>0$. For such a scheme $X$ we denote by $F=F_X$
the absolute Frobenius morphism of $X$. This is the identity on the topological space, and it takes a section $f$ of $\cO_X$ to $f^p$. Note that $F$ is a finite morphism of schemes
(not preserving the structure of schemes over $k$). We will also consider the iterates 
$F^e$ of $F$, for $e\geq 1$.

Let us recall some basic facts concerning pull-back of line bundles, sections, and subschemes.
Suppose that $L$ is a line bundle on $X$ and $Z$ is a closed subscheme of $X$.
\begin{enumerate}
\item[1)] There is a canonical isomorphism of line bundles $(F^e)^*(L)\simeq L^{p^e}$.
\item[2)]  The scheme-theoretic inverse image
$Z^{[e]}:=(F^e)^{-1}(Z)$ is a closed subscheme of $X$ defined by the ideal $I_Z^{[p^e]}$, such that if $I_Z$ is locally generated by $(f_i)_i$,
then $I_Z^{[p^e]}$ is defined by $(f_i^{p^e})_i$. In particular, if $Y$ is another closed subscheme of $X$,
having the same support as $Z$, there is some $e$ such that $Y$ is a subscheme of 
$Z^{[e]}$.
\item[3)] If $s\in H^0(Z,L\vert_Z)$, then $(F^e)^*(s)$ is a section in $H^0(Z^{[e]},(F^e)^*(L)\vert_{Z^{[e]}})$,
whose restriction to $Z$ gets identified with $s^{\otimes p^e}\in H^0(Z,L^{p^e}\vert_Z)$.
\end{enumerate}

\begin{lemma}\label{lem5}
If $X$ is a projective scheme over $k$ and $L$ is a line bundle on $X$, then 
\item[i)] $\SB(L)=\SB(L\vert_{X_{\rm red}})$.
\item[ii)] $\BB(L)=\BB(L\vert_{X_{\rm red}})$.
\end{lemma}

\begin{proof}
The inclusions ``$\supseteq$" in both i) and ii) follow from Lemma~\ref{lem1}. Let us prove the
reverse implication in i). Let $m$ be such that $\SB(L\vert_{X_{\rm red}})=
\Bs(L^m\vert_{X_{\rm red}})$.
Given $x\in X$, suppose that $x\not\in \Bs(L^m\vert_{X_{\rm red}})$.
Consider $s\in H^0(X_{\rm red},L^m\vert_{X_{\rm red}})$ such that $x\not\in Z(s)$. Let $J$
denote the ideal defining $X_{\rm red}$, and let
$e\gg 0$ be such that $J^{[p^e]}=0$. In this case $(F^e)^*(s)$ gives a section 
in $H^0(X,L^{mp^e})$ whose restriction to $X_{\rm red}$ is equal to $s^{\otimes p^e}$.
In particular, $x\not\in Z((F^e)^*(s))$. We conclude that $x\not\in \Bs(L^{mp^e})$,
hence $x\not\in\SB(L)$. This completes the proof of i).

Let $A$ be an ample line bundle on $X$, and let $m\gg 0$ be such that
$\BB(L\vert_{X_{\rm red}})=\SB(L^m\otimes A^{-1}\vert_{X_{\rm red}})$
and $\BB(L)=\SB(L^m\otimes A^{-1})$. The assertion in ii) now follows by applying i) to
$L^m\otimes A^{-1}$.
\end{proof}

The following is a key result from \cite{Keel}. We give a different proof, that has the advantage that it will apply also when replacing the stable base loci by the augmented base loci.

\begin{theorem}\label{thm_Keel} 
If $L$ is a nef line bundle on a projective scheme $X$, and $D$ is an effective Cartier divisor on $X$
such that $L(-D)$ is ample, then
$$\SB(L)=\SB(L\vert_D).$$
\end{theorem}

We isolate the key point in the argument in a lemma that we will use several times.

\begin{lemma}\label{key_lemma}
Let $A$ be an ample line bundle on a projective scheme $X$, and $D$ an effective Cartier
divisor on $X$. If $L:=A\otimes\cO_X(D)$ is nef, then for every $m\geq 1$ and every section $s\in H^0(D,L^m\vert_D)$, there is $e\geq 1$ such that $s^{\otimes p^e}\in H^0(D,L^{mp^e}\vert_D)$ is the restriction
of a section in $H^0(X,L^{mp^e})$.
\end{lemma}

\begin{proof}

Consider the short exact sequence
$$0\to L^m(-D)\to L^m\to L^m\vert_D\to 0.$$
Pulling-back by $F^e$ gives the exact sequence
$$0\to L^{mp^e}(-p^eD)\to L^{mp^e}\to L^{mp^e}\vert_{p^eD}\to 0.$$
Note that $L^m(-D)=L^{m-1}\otimes L(-D)$ is ample, since $L$ is nef and $L(-D)$ is ample.
By asymptotic Serre vanishing, we conclude that for $e\gg 0$ we have
$H^1(X, L^{mp^e}(-p^eD))=0$, and therefore the restriction map
$$H^0(X,L^{mp^e})\to H^0(X,L^{mp^e}\vert_{p^eD})$$
is surjective. Therefore there is $t\in H^0(X,L^{mp^e})$ such that $t\vert_{p^eD}=
(F^e)^*(s)$. In this case the restriction of $t$ to $D$ is equal to $s^{\otimes p^e}$.
\end{proof}

\begin{proof}[Proof of Theorem~\ref{thm_Keel}]
It follows from Lemma~\ref{lem1} that it is enough to show that if 
$P$ is a point on $X$ that does not lie in $\SB(L\vert_D)$, then $P$ does not lie in
$\SB(L)$. If $P$ does not lie on $D$, then it is clear that $P\not\in \SB(L)$, since
$A:=L\otimes\cO_X(-D)$ is ample. 
On the other hand, if
 $P\in D$, let
$m\geq 1$ be such that there is a section $s\in H^0(D,L^m\vert_D)$, with $P\not\in Z(s)$. Since 
$Z(s^{\otimes p^e})_{\rm red}=Z(s)_{\rm red}$, in order to show that $P\not\in\SB(L)$ it is enough to show that for some $e$,
the section $s^{\otimes p^e}$ lifts to a section in $H^0(X, L^{mp^e})$.
This is a consequence of Lemma~\ref{key_lemma}.
\end{proof}

\begin{corollary}\label{cor1_thm_Keel}
Let $X$ be a reduced projective scheme. 
If $L$ and $A$ are line bundles on $X$, with $A$ ample and $L$ nef, and 
 $Z=Z(s)$ for some $s\in
H^0(X,L\otimes A^{-1})$, then $\SB(L)=\SB(L\vert_Z)$.
\end{corollary}

\begin{proof}
Let $X'$ be the union of the irreducible components of $X$ that are contained in 
$Z$, and let $X''$ be the union of the other components (we consider on both $X'$ and $X''$
the reduced scheme structures). 
If $X'=X$, then $Z=X$ and there is nothing to prove, 
while if $X'=\emptyset$, then $Z$ is an effective Cartier divisor and the assertion follows from 
Theorem~\ref{thm_Keel}. Therefore we may and will assume that both $X'$ and $X''$ are non-empty.

Using the fact that $A$ is ample and the definition of the stable base locus, we obtain
$\SB(L)\subseteq\SB(L\otimes A^{-1})\subseteq Z$.  As in the proof of
Theorem~\ref{thm_Keel}, we see that it is enough to show that if $t\in H^0(Z,L^m\vert_Z)$
for some $m$, then there is $e\geq 1$ such that $t^{\otimes p^e}$ can be lifted to a
section in $H^0(X,L^{mp^e})$.  By applying Lemma~\ref{key_lemma} to $X''$, $D=Z\cap X''$
and the ample line bundle $L\vert_{X''}\otimes\cO_{X''}(-D)$, we see that for some $e$ we
can lift $t^{\otimes p^e}\vert_{X''\cap Z}$ to a section $t''\in H^0(X'',
L^{mp^e}\vert_{X''})$. Since $X'\subseteq Z$, the restriction of $t''$ to $X''\cap X'$ is
equal to $t^{\otimes p^e}\vert_{X'\cap X''}$. Therefore we can glue $t^{\otimes
  p^e}\vert_{X'}$ with $t''$ to get a section in $H^0(X,L^{mp^e})$ lifting $t^{\otimes
  p^e}$.
\end{proof}

Recall that if $L$ is a nef line bundle on  the projective scheme $X$, then the
\emph{exceptional locus} $L^{\perp}$ is the union of all closed irreducible subsets
$V\subseteq X$ such that $L\vert_V$ is not big. Since $L$ is nef, this condition is equivalent
to the fact that $\dim(V)>0$ and $(L\vert_V^{\dim(V)})=0$.

\begin{remark}\label{rem1}
It is easy to see by induction on $\dim(X)$ that $L^{\perp}$ is a closed subset of $X$. 
Note first that if
$X_1,\ldots,X_r$ are the irreducible components of $X$ (with the reduced scheme structures), then clearly
$L^{\perp}=(L\vert_{X_1})^{\perp}\cup\ldots\cup (L\vert_{X_r})^{\perp}$. Therefore we may assume that $X$ is integral.
In this case, if $L$ is not big, then $L^{\perp}=X$. Otherwise, 
we can find an effective Cartier divisor $D$ and a positive integer $m$ such that
$L^m(-D)$ is ample. It is clear that if $L\vert_V$ is not big, then $V\subseteq D$. Therefore
$L^{\perp}=(L\vert_D)^{\perp}$, hence it is closed by induction.
\end{remark}

The following result is one of the main results from \cite{Keel}. As we will see, this is an easy consequence of Corollary~\ref{cor1_thm_Keel}.

\begin{corollary}\label{cor2_thm_Keel}
If $L$ is a nef line bundle on the projective scheme $X$, then 
$\SB(L)=\SB(L\vert_{L^{\perp}})$.
\end{corollary}

\begin{proof}
Arguing by Noetherian induction, we may assume that the result holds for every proper
closed subscheme of $X$. Since $L^{\perp}=(L\vert_{X_{\rm red}})^{\perp}$, it follows from
Lemma~\ref{lem5} that we may assume that $X$ is reduced. If the restriction of $L$ to
every irreducible component of $X$ is not big, then $L^{\perp}=X$, and there is nothing to
prove.  From now on we assume that this is not the case, and let $X'$ and $X''$ be the
union of those irreducible components of $X$ on which the restriction of $L$ is not
(respectively, is) big. On both $X'$ and $X''$ we consider the reduced scheme
structures. Note that by assumption $X''$ is nonempty.

Consider an ample line bundle $A$ on $X$. 
It follows from Lemma~\ref{lem3} that if $m\gg 0$, there is a section $s\in H^0(X,L^m\otimes A^{-1})$ such that no irreducible component of $X''$ is contained in $Z=Z(s)$
(but such that $X'\subseteq Z$).
It is clear that if $V$ is an irreducible  closed subset of $X$ such that
$L\vert_V$ is not big, then $V\subseteq Z$. Therefore $L^{\perp}=
(L\vert_Z)^{\perp}$. Since $X''$ is nonempty, it follows that $Z\neq X$, hence the inductive assumption gives $\SB(L\vert_Z)=\SB(L\vert_{L^{\perp}})$.
On the other hand, Corollary~\ref{cor1_thm_Keel} gives 
$$\SB(L)=\SB(L^m)=\SB(L^m\vert_Z)=\SB(L\vert_Z),$$
which completes the proof.
\end{proof}

We can now prove the second theorem stated in the Introduction.

\begin{proof}[Proof of Theorem~\ref{thm_main2}]
We suitably modify the argument in the proof of Theorem~\ref{thm_Keel}. By Lemma~\ref{lem1},
it is enough to prove the inclusion $\BB(L)\subseteq\BB(L\vert_D)$. Furthermore,
 Lemma~\ref{lem5}  implies $\BB(L\vert_D)=\BB(L\vert_{2D})=\BB(L^2\vert_{2D})$
 and we have $\BB(L)=\BB(L^2)$, hence we may replace $L$ by $L^2$ and $D$ by $2D$ to assume that $L(-D)\simeq A^2$, for some
 ample line bundle $A$.
 
 Suppose that $P$ is a point that does not lie on $\BB(L\vert_D)$.
 If $P\not\in D$, since $L(-D)$ is ample, it follows that $P\not\in\BB(L)$. Hence from now on we may assume that $P\in D$. By assumption,
for $m\gg 0$ we have
 $P\not\in\SB(L^m\otimes A^{-1}\vert_D)$. Let us choose $r\geq 1$ such that there is
 $t\in H^0(D, L^{rm}\otimes A^{-r}\vert_D)$ with $P\not\in Z(t)$. Furthermore, since we may take
  $r$ large enough, we may assume that $A^{r-1}\vert_{D}$ is globally generated. Let
  $t'\in H^0(D, A^{r-1}\vert_D)$ be such that $P\not\in Z(t')$. Therefore $t\otimes t'\in
  H^0(D, L^{rm}\otimes A^{-1})$ is such that $P\not\in Z(t\otimes t')$. Note that
  $L^{rm}\otimes A^{-1}(-D)\simeq L^{rm-1}\otimes A$ is ample, since $L$ is nef and
  $A$ is ample. Therefore Lemma~\ref{key_lemma} implies that for some $e\geq 1$, the section $t^{\otimes p^e}\otimes {t'}^{\otimes p^e}$ 
 can be lifted to a section in $H^0(X, L^{rmp^e}\otimes A^{-p^e})$, and this section clearly
 does not vanish at $P$. This shows that $P\not\in\BB(L)$, and completes the proof of the theorem.
\end{proof}

\begin{corollary}\label{cor_thm_main2}
Let $X$ be a reduced projective scheme. If $L$ and $A$ are line bundles on $X$, with $L$ nef and
$A$ ample, and $Z=Z(s)$ for some $s\in H^0(X,L\otimes A^{-1})$, then
$\BB(L)=\BB(L\vert_Z)$.
\end{corollary}

\begin{proof}
We modify slightly the argument in the proof of Theorem~\ref{thm_main2}, along the lines in the proof of Corollary~\ref{cor1_thm_Keel}. By Lemma~\ref{lem1}, it is enough to show that if
$P\not\in\BB(L\vert_Z)$, then $P\not\in\BB(L)$. Let $X'$ be the union of the irreducible components of $X$ that are contained in $Z$, and $X''$ the union of the other components,
both considered with the reduced scheme structures. If $X'=X$, then $Z=X$ and there is nothing to prove, while if $X'=\emptyset$, then $Z$ is an effective Cartier divisor, and the assertion follows from Theorem~\ref{thm_main2}. From now on, we assume that both $X'$ and $X''$ are nonempty.

After replacing $L$ and $A$ by $L^2$ and $A^2$, respectively, and $s$ by $s^{\otimes 2}$,
we may assume that $A=B^2$, for some ample line bundle $B$
(note that $\BB(L\vert_{Z(s)})=\BB(L\vert_{Z(s^{\otimes 2})})$ by Lemma~\ref{lem5}). 
Suppose that $P\not\in\BB(L\vert_Z)$. If $P\not\in Z$, then $P\not\in\SB(L\otimes
A^{-1})$; since $A$ is ample, we have $\BB(L)\subseteq \SB(L\otimes A^{-1})$, hence
$P\not\in \BB(L)$.  From now on we assume that $P$ lies in $Z$.

Arguing as in the proof of Theorem~\ref{thm_main2}, we find a section 
$$t\otimes t'\in H^0(Z, L^{rm}\otimes A^{-1}\vert_Z)$$
such that $P\not\in Z(t\otimes t')$, and we use Lemma~\ref{key_lemma} to deduce that for some $e\geq 1$, we can lift
$t^{\otimes p^e}\otimes {t'}^{p^e}\vert_{Z\cap X''}$ to a section 
$t''\in H^0(X'',L^{rmp^e}\otimes A^{-p^e}\vert_{X''})$. 
Recall that $X'\subseteq Z$, hence $X'\cap X''\subseteq Z\cap X''$, and therefore
$t''\vert_{X'\cap X''}=t^{\otimes p^e}\otimes {t'}^{\otimes p^e}\vert_{X'\cap X''}$.
Since $X$ is reduced, it follows that we can glue $t''$ and $t^{\otimes p^e}\otimes {t'}^{\otimes p^e}\vert_{X'}$ to a section in $H^0(X,L^{rmp^e}\otimes A^{-p^e})$ that does not vanish at $P$.
Therefore $P\not\in\BB(L)$, which concludes the proof.
\end{proof}

We now give the proof of the characteristic $p$ version of Nakamaye's theorem.

\begin{proof}[Proof of Theorem~\ref{thm_main}]
We argue as in the proof of Corollary~\ref{cor2_thm_Keel}. By Noetherian induction,
we may assume that the theorem holds for every proper closed subscheme of $X$.
Lemma~\ref{lem5} implies $\BB(L)=\BB(L\vert_{X_{\rm red}})$, and since $L^{\perp}=
(L\vert_{X_{\rm red}})^{\perp}$, we may assume that $X$ is reduced.

Note that the inclusion $L^{\perp}\subseteq\BB(L)$ is clear: if $V$ is a closed irreducible subset of $X$
that is not contained in $\BB(L)$, then we can find an ample line bundle $A$, a positive
integer $m$, and
$s\in H^0(X,L^m\otimes A^{-1})$ such that $V\not\subseteq Z(s)$. Therefore
$s\vert_V$ gives a nonzero section of $L^m\otimes A^{-1}\vert_V$, hence $L\vert_V$
is big. This shows that  it is enough to prove the inclusion $\BB(L)\subseteq L^{\perp}$. 

If the restriction of $L$ to all the irreducible components of $X$ is not big, then $L^{\perp}=X$, and the
assertion is clear. Otherwise, let $X'$ denote the union of the irreducible components of $X$
on which the restriction of $L$ is not big, and $X''$ the union of the other components, both with the reduced scheme structures.  It follows from 
Lemma~\ref{lem3} that 
given any ample line bundle $A$, we can find $m\geq 1$ and a section
$s\in H^0(X,L^m\otimes A^{-1})$ whose restriction to every component of $X''$ is nonzero
(and whose restriction to $X'$ is zero). Let $Z=Z(s)$. By assumption $X''$ is nonempty, and therefore
$Z$ is a proper closed subscheme of $X$, hence by the inductive assumption
we have $\BB(L\vert_Z)=(L\vert_Z)^{\perp}$. If $V\subseteq X$ is an irreducible closed subset such that $L\vert_V$ is not big,
then $V\subseteq Z$, hence $L^{\perp}=(L\vert_Z)^{\perp}$. 
On the other hand, Corollary~\ref{cor_thm_main2}  
gives $\BB(L)=\BB(L\vert_Z)$, and we conclude that $\BB(L)=L^{\perp}$.
\end{proof}

\providecommand{\bysame}{\leavevmode \hbox \o3em
{\hrulefill}\thinspace}


\begin{thebibliography}{Nakamaye}


\bibitem{ELMNP}
 L. Ein, R. Lazarsfeld, M.
Musta\c{t}\v{a}, M. Nakamaye and M. Popa, Asymptotic
invariants of base loci,  Ann. Inst. Fourier (Grenoble) \textbf{56} (2006), 
1701--1734.

\bibitem{Keel}
S.~Keel,  Basepoint freeness for nef and big line bundles in positive characteristic,
 Ann. of Math. (2) \textbf{149} (1999),  253--286.
 
 
\bibitem{positivity}
R.~Lazarsfeld, \emph{Positivity in algebraic geometry} II, Ergebnisse der Mathematik und ihrer
Grenzgebiete, 3. Folge, Vol. 49, Springer-Verlag, Berlin, 2004.


 
 \bibitem{Nakai}
 Y.~Nakai, Some fundamental lemmas on projective schemes, Trans. Amer. Math. Soc.
 \textbf{85} (1963), 296--302.
 
 \bibitem{Nakamaye}
 M.~Nakamaye, Stable base loci of linear series, Math. Ann. \textbf{318} (2000),  837--847.




\end{thebibliography}
\end{document}